\documentclass{article}
\usepackage{tikz}
\usetikzlibrary{patterns}
\usepackage{graphicx,amsmath,amsfonts,amsthm}
\newtheorem{proposition}{Proposition}[section]

\newtheorem{theorem}[proposition]{Theorem}
\theoremstyle{definition}
\newtheorem{remark}[proposition]{Remark}

\newcommand{\dt}{\mathrm{\Delta}t}
\newcommand{\pder}[2]{\frac{\partial #1}{\partial #2}}
\newcommand{\diag}{\mathsf{diag}}
\newcommand{\tridiag}{\mathsf{tridiag}}
\newcommand{\N}{\mathbb{N}}

\newcommand{\tr}{\mathsf{T}}

\renewcommand{\vec}[1]{\mathbf{#1}}

\title{Preconditioned fully implicit PDE solvers for monument conservation}
\author{Matteo Semplice%
\thanks{Dipartimento di Fisica e Matematica,
Universit\`a dell'Insubria - Sede di Como, Via Valleggio 11, 22100
Como, Italy, \textbf{E-mail:} {\sl matteo.semplice@uninsubria.it}
}}
 
\date{\today}

\begin{document}
\maketitle

\begin{abstract}
  Mathematical models for the description, in a quantitative way, of
  the damages induced on the monuments by the action of specific
  pollutants are often systems of nonlinear, possibly degenerate,
  parabolic equations. Although some the asymptotic properties of the
  solutions are known, for a short window of time, one needs a
  numerical approximation scheme in order to have a quantitative
  forecast at any time of interest.

  In this paper a fully implicit numerical method is proposed,
  analyzed and numerically tested for parabolic equations of porous
  media type and on a systems of two PDEs that models the sulfation of
  marble in monuments.  Due to the nonlinear nature of the underlying
  mathematical model, the use of a fixed point scheme is required and
  every step implies the solution of large, locally structured, linear
  systems. A special effort is devoted to the spectral analysis of the
  relevant matrices and to the design of appropriate iterative or
  multi-iterative solvers, with special attention to preconditioned
  Krylov methods and to multigrid procedures. Numerical experiments
  for the validation of the analysis complement this contribution.
\end{abstract}

\section{Introduction}
The problem of monitoring, preserving and, when needed, restoring
monuments and works of art has become more and more relevant in recent
years for the conservation of our cultural heritage, after the
recognition of the negative effects of some pollutants on the
monuments. Numerous studies made researchers and restorers more and
more aware that gaseous pollutants, atmospheric particulate matter,
and some microorganisms can adversely affect the status of our
monuments.  In order to monitor the cultural heritage and for
precisely programming the restoration works, it is of paramount
importance to be able to accurately assess the status of each
monument. Along these lines, quantitative methods are emerging and
making their way into the practice of preservation and restoration.
These have the obvious advantage of allowing fair comparison of the
state of different monuments, supporting the decision process on what
to restore, clean, etc and on the relative urgency of each case.

As an example, consider the ``black crusts'' that grow on marble
surfaces as an effect of sulfation of the carbonate stone that is
turned into gypsum when reacting with $\mathrm{SO_2}$ in a moist
environment. Since urban concentrations of $\mathrm{SO_2}$ can be
nowadays more than $100$ times higher than the atmospheric basal
values, this effect has become very important in the last
decades. Sulfation can cause permanent damage to the monuments because
gypsum crusts can be easily eroded by rain or (when located in
protected areas) can become unaesthetically black including
particulate matter from the atmosphere and eventually exfoliate
\cite{Hay82,GKPC89,BLR00}.

A better scheduling of cleaning or deeper restoration can be devised
if the thickness (and composition) of the crust can be forecast in
quantitative way, providing a way to compute and thus predict the time
evolution of the crust.  The most common quantitative evaluation of
the sulfation phenomena that is used in practice consists in assuming
that the thickness is directly proportional to the length of time of
the exposure to the pollutants, with a proportionality coefficient
obtained by fitting data from a large number of monuments
\cite{Lip89}.  Although this may give an average indication good
enough for civil buildings, the uniqueness and cultural importance of
a work of art calls for a more detailed analysis, that can take into
account the local environment to which the monument is exposed.

A mathematical model of the sulfation of marble based on the chemical
reactions involved was developed by Natalini and coworkers
\cite{ADN:sulfation, GN05} at IAC-CNR (Rome) and tested against
experiments, see \cite{GSNF08}.

It is worthwhile to remark that the mathematical model is able to
provide new information, which partly contradicts the most common
quantitative evaluation methods based on data fitting. In particular
(see \cite{GN07}) the asymptotic study of the equations in a one
dimensional setting reveals that for large times the thickness of the
gypsum crust does not grow proportionally to the elapsed time as in
the Lipfert formula, but proportionally to its square root: the speed
of growth of the crust is significantly reduced as time goes
on. Clearly this means that a complete removal of the crust will speed
up the damage and calls for study of optimal strategies for the
periodic partial crust removal.

However the asymptotic analysis does not give enough information on
what happens for short times and moreover the study is not yet
available for complex geometries. For example on a corner stone,
$\mathrm{SO_2}$ penetrates the marble from two sides: how does the
crust grow? Does it get rounded? How much? And, more importantly, what
about the fine particulars of decorations or statues? In some cases
sulfation caused an almost complete loss of details: can the model
predict the thickness of the crust there and allow the scheduling of
an optimal conservation strategy?  In order to answer the previous
questions, we need a numerical method to solve the equations of the
model developed by the group in Rome (see \cite{ADN:sulfation}).  This
is a system of two equations, one of which is nonlinear of parabolic
type.

In this paper we generalize and apply novel numerical techniques
studied in \cite{SDS:degdiff} to integrate for long times nonlinear,
possibly degenerate, parabolic equations like those appearing in the
model by \cite{ADN:sulfation}. We wish to point out that the
techniques developed here have applications that go beyond the
aforementioned model. For example, in the area of {\em planned
  conservation}, they could be adapted to numerically investigate the
more complete sulfation model described in \cite{AFNT07:model} and the
consolidation model presented in \cite{CGNNS:teos}.

In the literature, degenerate parabolic equations have been
discretized mainly using explicit or semi-implicit methods, thus
avoiding to solve the nonlinear equation arising from the elliptic
operator.  A remarkable class of methods arise directly from the
so-called non-linear Chernoff formula \cite{BP72} for time
advancement, coupling it with a spatial discretization: for finite
differences this was started in \cite{BBR79} and for finite elements
by \cite{MNV87}.  For example the numerical scheme analysed in \cite
{ADN:sulfation} for integrating the sulfation model belongs to the
class of semi-implicit methods.  More recently, another class related
to the relaxation approximation emerged: such numerical procedures
exploit high order non-oscillatory methods typical of the
discretization of conservation laws and their convergence can be
proved making use of semigroup arguments similar to those relevant for
proving the Chernoff formula \cite{CNPS07:degdiff}.

\newcommand{\Lop}{\mathcal{L}}

In this paper we consider two fully-implicit discretizations in
time, thus solving a nonlinear system at each time step. In order to
fix ideas, consider the parabolic equation
\begin{equation}\label{eq:pm}
\pder{u}{t} = \nabla\cdot\left(D(u)\nabla u\right),
\end{equation}
where $D(u)$ is a non-negative differentiable function and denote with
$\Lop_D(u)$ the elliptic operator on the right hand side. Considering
a time discretization such that $\dt=t^n-t^{n-1}$ and denoting with
$U$ the numerical solution, we employ the 
(first order accurate) Implicit Euler scheme
\begin{equation}\label{eq:EI}
U(t^n,x)-\dt \Lop_D(U(t^n,x)) = U(t^{n-1},x)
\end{equation}
and the (second order accurate) Crank-Nicholson scheme
\begin{equation}\label{eq:CN}
U(t^n,x)-\frac\dt2 \Lop_D(U(t^n,x)) = 
U(t^{n-1},x) +\frac\dt2 \Lop_D(U(t^{n-1},x))
\end{equation}
(Note that \eqref{eq:EI} is also known as the Crandall-Liggett
formula, after \cite{CL71}.)

The computation of $U(t^n,x)$ with \eqref{eq:EI} or \eqref{eq:CN}
requires to solve a nonlinear equation whose form is determined by the
elliptic operator and the nonlinear function $D(u)$, but the
convergence is guaranteed without restrictions on the time step $\dt$.
Due to the nonlinear nature of the underlying mathematical model, the
use of a fixed point scheme is required and the choice of the faster
Newton-like methods implies the solution at every step of large,
locally structured (in the sense of Tilli, see \cite{Tilli98} and
\cite{Serra06:glt}) linear systems. A special effort was devoted in
\cite{SDS:degdiff} to the spectral analysis of the relevant matrices
and to the design of appropriate iterative or multi-iterative solvers
(see \cite{Serra93:multi}), with special attention to preconditioned
Krylov methods and to multigrid procedures (see
\cite{Greenbaum:book,Saad:book,Hackbusch:book,Trottenberg:book} and
references therein for a general treatment of iterative solvers). In
this paper we will argue that those methods can be extended to the
case of systems and perform numerical tests on the model of
\cite{ADN:sulfation}.

The paper is organised as follows.  In section \ref{sec:scalar} we
recall the results of \cite{SDS:degdiff} on scalar equations, and
extend them to the case of a scheme which is second order in time. In
Section \ref{sec:sulfation} we introduce the sulfation model, the
implicit numerical schemes and the preconditioners for the linear
systems. Both sections are complemented by numerical experiments in
one and two spatial dimensions. Finally in Section
\ref{sec:conclusions} we point out some possible developments of this
work.

\section{Scalar equations}
\label{sec:scalar}
In this section, we consider the case of a single equation of the
porous media type, namely \eqref{eq:pm}, where $D(u)$ is a
non-negative differentiable function.  The parabolic equation is of
degenerate type whenever $D(u)$ vanishes for some values of $u$. For
the convergence analysis of our numerical methods, we will require
that $D(u)$ is at least continuously differentiable, while the
existence of solutions is guaranteed under the milder assumption of
continuity \cite{VazquezBOOK}. Most applications of the porous media
equation involve $D(u)=u^m$ for some positive $m$.

For this particular choice, the following self-similar exact solutions
have been computed by Barenblatt and Pattle (see \cite{VazquezBOOK}):
\begin{equation}
  \label{eq:barenblatta}
  u(t,\vec{x}) = 
  t^{-\alpha}\left[
    1-k\left(\frac{|\vec{x}|}{t^{\alpha/d}}\right)^2
  \right]_+^{\frac1{m-1}}
  \qquad
  \text{for } t>0, \vec{x}\in\mathbb{R}^d
\end{equation}
where $|\vec{x}|=\sqrt{\sum_1^d x_i^2}$ and
\(\alpha=\tfrac{d}{d(m-1)+2}\), \(k=\alpha\tfrac{m-1}{2md} \).  These
solutions are singular for $t=0$, but for $t>0$ represent important
reference cases both for the analysis of the solutions of
\eqref{eq:pm} and for numerical tests.

Here, results of \cite{SDS:degdiff} on scalar equations are recalled
and extended to the case of a scheme which is second order in time.
They will be used in section \ref{sec:sulfation}, which deals with a
system of two equations, since the linear systems arising there
include, as a subsystem, those considered in this section.

\subsection{Numerical scheme}
\label{ssec:scalar1}
In order to obtain a numerical scheme for approximating the solutions
of \eqref{eq:pm}, we first discretize the time variable with the
Crank-Nicholson formula \eqref{eq:CN}, generalising the simpler case
of Implicit Euler that was considered in \cite{SDS:degdiff}.  We point
out that searching for high order schemes for equation \eqref{eq:pm}
would seem at first useless, since the exact solution of the equation
are in general continuous but not differentiable, so any scheme would
converge in theory with order $1$.  However, in practice often the
exact solution is piecewise regular, allowing an higher order scheme
to converge faster than first order and in any case to achieve better
errors than \eqref{eq:EI} at a given spatial resolution, even if the
theoretical order of convergence is not reached.

We then complete the discretization by considering the points
$x_k=a+kh$ in the spatial domain $[a,b]$, where $h=(b-a)/(N+1)$ and
$k=0,\ldots,N+1$, and approximating the one-dimensional Laplacian
operator with the usual 3-point finite difference formula, i.e.
\begin{multline}\label{eq:FDlaplacian}
\left.\pder{}{x}\left(D(u)\tfrac{\partial u}{\partial x}
\right)\right|_j 
= \frac{D(u)_{j+1/2} \left.\tfrac{\partial u}{\partial x}\right|_{j+1/2}
-D(u)_{j-1/2} \left.\tfrac{\partial u}{\partial x}\right|_{j-1/2}}{h}+o(1)\\
=\frac{D(u)_{j+1/2} (u_{j+1}-u_j)
-D(u)_{j-1/2} (u_j-u_{j-1})}{h^2}+o(1)\\
=\frac{(D(u_{j+1})+D(u_j)) (u_{j+1}-u_j) -(D(u_j)+D(u_{j-1}))
(u_j-u_{j-1})}{2h^2}+o(1)
\end{multline}

In order to write down compactly the equations for the numerical
scheme, we collect in a vector $\vec{u}^n$ all the unknown values
$u^n_j=u^n(x_j)$. For example when Dirichlet boundary conditions are
considered, since $u_0$ and $u_{N+1}$ are known, $\vec{u}^n$ has $N$
elements, namely $u^n_1,\ldots,u^n_N$.

We denote by $\tridiag^N_k[\beta_k, \alpha_k,
\gamma_k]$ a square tridiagonal matrix of order $N$ with entries
$\beta_k$ on the lower diagonal, $k=2,\cdots,N$, $\alpha_k$ on the
main diagonal, $k=1,\cdots,N$, and $\gamma_k$ on the upper diagonal,
$k=1,\cdots,N-1$.
We also denote with
$\diag^N[\alpha_k]$ the $N\times{N}$ square diagonal matrix with
$\alpha_k$ on the $k^{\text{th}}$ row.
With this notation, recalling that \eqref{eq:CN} is a second order
approximation, 
\begin{equation}
\label{eq:CNfd}
\vec{u}^n - \vec{u}^{n-1} = \frac12\frac{\dt}{h^2}
L_{D(\vec{u}^n)}\vec{u}^n 
\frac12\frac{\dt}{h^2} L_{D(\vec{u}^{n-1})}\vec{u}^{n-1} + o(\dt^2+h^2)
\end{equation}
where
\begin{equation}
\label{eq:Ld}
L_{D(\vec{u})} =\tridiag^N_k [D_{k-1/2},
-D_{k-1/2}-D_{k+1/2}, D_{k+1/2}] 
\end{equation}
and
\begin{equation*}
D_{j+1/2}=\frac{D(u_{j+1})+D(u_{j-1})}2
\,,\; j=0,\ldots,N
\end{equation*}

In two dimensions, on a finite grid composed by the $(N+2)\times(N+2)$
points $\vec{x}_{i,j}=a+ih\vec{e}_1+jh\vec{e}_2$, where $a$ is the
lower left corner of the domain, $h$ the discretization parameter,
$\vec{e}_l$ ($l=1,2$) unit vectors along the coordinate axis and
$i,j\in\N$, the matrix $L_{D(\vec{u})}$ approximating
$\nabla\cdot\left(D(u^n)\nabla u^n\right)$ is pentadiagonal. When
considering Dirichlet boundary conditions, adopting the usual
lexicographic ordering of the unknowns $u^n_{i,j}$, $L_{D(\vec{u})}$
is a $N^2\times{N^2}$ square matrix and nonzero entries can be found
only on the main diagonal, on the $1^{\text{st}}$ and $N^{\text{th}}$
upper and lower diagonal.

Finally, we point out that the asymptotic spectral properties of the
matrices arising from this discretization ($L_D$ in our case), to a
large extent, do not depend on the choice of the finite difference
formula, but really depend on the Locally Toeplitz structure that in
turn arises from operator appearing in the PDE
(\cite{Serra06:glt}). Thus it should be possible to generalize most of
the results of the following sections on linear solvers and
preconditioning to other spatial discretizations techniques, including
finite element methods.

\subsection{Newton method} 
\label{ssec:newton}
In order to advance the numerical solution from $\vec{u^{n-1}}$ to
$\vec{u^n}$, the nonlinear system of equations \eqref{eq:CNfd} must be
solved at each timestep. We achieve this, by iterating with the
Newton's method for the function
\begin{equation}\label{eq:CNF}
  F(\vec{u}) = \vec{u} 
  -\frac12\frac{\dt}{h^2} L_{D(\vec{u})}\vec{u}
  -\frac12\frac{\dt}{h^2} L_{D(\vec{u}^{n-1})}\vec{u}^{n-1}
  - \vec{u}^{n-1} 
\end{equation}
The Jacobian of $F$ is
\begin{align} 
&F^\prime(\vec{u}) = 
X_N(\vec{u}) 
+ Y_N(\vec{u})\label{Jf2}
\\
&X_N(\vec{u}) = 
I_N-\frac12\frac{\dt}{h^2}L_{D(\vec{u})} \label{Xn2}
\\
&{Y}_N(\vec{u}) = 
-\frac12\frac{\dt}{h^2}T_N(\vec{u}) \diag^N_k(D^\prime_k)\label{Yn2}
\\
&T_N(\vec{u}) =
\tridiag^N_k[u_{k-1}-u_k,u_{k-1}-2u_k+u_{k+1},u_{k+1}-u_k]\label{Tn}
\end{align}
where $T_N$ is the same matrix of the first order case
(e.g. \eqref{Tn} in one spatial dimension).
The only difference in two dimensions is that $T_N$ is pentadiagonal.

We observe that the first order scheme based on Implicit Euler gives
rise to
\begin{equation}\label{eq:F:IE}
  \widetilde{F}(\vec{u}) = \vec{u} 
  -\frac{\dt}{h^2} L_{D(\vec{u})}\vec{u}
  - \vec{u}^{n-1} 
\end{equation}
This is the case studied in \cite{SDS:degdiff}. Since the Jacobian
matrix of $\widetilde{F}$ differs from $F^\prime$ only for the missing
$\tfrac12$ factors in $X_N$ and $Y_N$, most of the results proved in
\cite{SDS:degdiff} can be adapted to the present setting.  In the
following we will thus only sketch the proofs.

Our main result is that the Newton method defined by ${F}$,
initialised with $\vec{u}^{n,0}=\vec{u}^{n-1}$, is convergent under a
linear restriction on the timestep. In order to prove it, we need the
following estimate for the norm of the inverse of
$J$.
\begin{proposition}\label{prop:Fprimo}
  Consider $F(\vec{u})$ as defined in \eqref{eq:CNF}, where
  $\vec{u}$ is a sampling (at a given time $t$) of a solution $u$ of
  \eqref{eq:pm} with $D$ differentiable and having first derivative
  Lipschitz continuous.  If, in addition, $\vec{u}$ is differentiable
  with Lipschitz continuous first derivative, we have that
  \begin{equation}\label{eq:Finvnorminf}
    \left\|
    F^\prime(\vec{u})^{-1}\right\|_\infty \leq C_1
  \end{equation}
  for $h$ sufficiently small and under the additional assumption that
  $\dt\le C_\infty h$ for some $C_\infty >0$ that does not depend on $h$. 
\end{proposition}
\begin{proof}
  $F^\prime(\vec{u})$ differs from $\widetilde{F}^\prime(\vec{u})$ for
  $\widetilde{F}$ defined in \eqref{eq:F:IE} only by factors
  $\tfrac12$ appearing before any $D_{k+1/2}$ term. Since these terms
  are discarded in the estimates for the proof of the analogous result
  for $F$, the same proof is valid here. See \cite{SDS:degdiff} for
  the details.
\end{proof}

The following result is a classical tool (see \cite{Ortega:book}) for
handling the global convergence of the Newton procedure.
\begin{theorem}[Kantorovich]\label{th:kan}
  Consider the Newton method for approximating the zero of a vector
  function $F(\vec{u})$, starting from the initial approximation
  $\vec{u}^{(0)}$. Under the assumptions that
  \begin{subequations}\label{eq:hp}
    \begin{align}
      &\|\left[F^{\prime}(\vec{u}^{(0)})\right]^{-1}\| \leq\beta\, ,
      \label{eq:hp:beta}\\
      &\|\left[F^{\prime}(\vec{u}^{(0)})\right]^{-1}F(\vec{u}^{(0)})\|
      \leq\eta\, ,
      \label{eq:hp:eta}\\
      &\|F^{\prime}(\vec{u})-F^{\prime}(\vec{v})\|\leq\gamma\|\vec{u}-\vec{v}\|\,
      , \label{eq:hp:gamma}
    \end{align}
  \end{subequations}
  and that
  \begin{equation}\label{eq:12}
    \beta\eta\gamma < \frac12\, ,
  \end{equation}
  the method is convergent and, in addition, the stationary point of
  the iterations lies in the ball with centre $\vec{u}^{(0)}$ and
  radius
  \[\frac{1-\sqrt{1-2\beta\eta\gamma}}{\beta\gamma}.\]
\end{theorem}

\begin{theorem}\label{teor:newton}
  The Newton method for $\widetilde{F}(\vec{u})$ defined in
  \eqref{eq:CNF} for computing $\vec{u}^{n}$ is convergent when
  initialised with the solution at the previous timestep (i.e.
  $\vec{u}^{n,0}=\vec{u}^{n-1}$) and for $\dt\leq C h$, for a positive
  constant $C$ independent of $h$.
\end{theorem}
\begin{proof}
  The proof of the same statement for $F$ as given in \cite{SDS:monum}
  can be easily adapted to the present case. The technique is to first
  establish estimates \eqref{eq:hp} as follows:
  \begin{itemize}
    \item $\beta\leq C_1$ if $\dt\leq C_{\infty}h$ by Proposition
      \ref{prop:Fprimo}
    \item $\eta\leq\beta C_2\dt$ by first applying  Proposition
      \ref{prop:Fprimo} again and then estimating
      \[ \left\|\widetilde{F}(\vec{u}^{n,0})\right\|_\infty=
      \left\| \widetilde{F}(\vec{u}^{n-1})\right\|_\infty=
      \dt \left\|L_{D(\vec{u}^{n-1})}\vec{u}^{n-1}\right\|_\infty=
      O(\dt)
      \]
    \item $\gamma\leq8\|D^\prime\|_\infty\tfrac{\dt}{h^2}$ by direct
      computation as in \cite{SDS:degdiff}.
    \end{itemize}
    This implies that condition \eqref{eq:12} can be satisfied when
    choosing $\dt\leq Ch$ for a sufficiently small positive constant
    $C$ that is independent of $h$.
  \end{proof}
  
  \begin{remark} Setting the initial guess with the average between
    $\vec{u}^{n-1}$ and the value given by an Explicit Euler step, like
      \[
      \vec{u}^{n,0} = \vec{u}^{n-1} +\frac12 \frac{\dt}{h^2}
      L_{D(\vec{u}^{n-1})}\vec{u}^{n-1}
      \]
      does not change the convergence ratio, but in practice one needs
      less iterations to reach a given tollerance.
  \end{remark}

\subsection{Iterative methods for the linear system}
\label{ssec:alglin}
Of course the Jacobian matrix $F^\prime(\vec{u}^{n,(s)})$ is not
explicitly inverted at each Newton step, but instead we compute the
$(s+1)^{\text{th}}$ Newton iterate by first solving the linear system
\[ F^\prime(\vec{u}^{n,(s)}) \vec{v}^{(s)}
= F(\vec{u}^{n,(s)}) \]
for $\vec{v}^{(s)}$ and then setting
\[ \vec{u}^{n,(s+1)} = \vec{u}^{n,(s)} + \vec{v}^{(s)}\]

The matrix $A_N=F^\prime(\vec{u}^{n,(s)})$ is a square tridiagonal
(respectively pentadiagonal) $N\times{N}$ (respectively
$N^2\times{N^2}$) matrix when the domain is one (respectively two)
dimensional. Its spectral properties are crucial in choosing an
appropriate solver for the linear system.  Given the large dimension
of the system, we aim at an iterative method with an optimal
preconditioner, so that we can compute $\vec{v}^{(s)}$, on average, in
a finite number of iterations.

Moreover $A_N$ differs from $\widetilde{F}^\prime$ only by the factors
$\tfrac12$ that were missing in the matrices studied in
\cite{SDS:degdiff}. It can thus be shown that $A_N$ is not symmetric,
but it is dominated by its symmetric part $(A_N+A_N^\tr)/2$, which is
in turn essentially a weighted laplacian.  Rather detailed information
on the spectrum of $A$ can be gained via the theory of Locally
Toeplitz Sequences of \cite{Tilli98}. In particular, when $\dt$ is
chosen proportional to $h$, the sequence of $N\times N$ matrices
$\{hA_N\}$ obtained for increasing number of grid points is Locally
Toeplitz (in the sense of \cite{Tilli98}) with respect to the pair of
functions $\left(D(u(x)),2-2\cos(s)\right)$ defined on
$[a,b]\times[0,2\pi]$. Hence (see \cite{SDS:degdiff}) we can expect
the GMRES method, which is picked due to the asymmetry of $A_N$ as the
main iterative solver, to converge in $O(\sqrt{N})$ iterations.

In order to study a preconditioning strategy, first observe that the
sequence $\{X_N\}$ of the symmetric parts of $A_N$ is also Locally
Toeplitz with respect to the same generating functions. Next recall
that any Locally Toeplitz sequence is also Generalized Locally
Toeplitz, a class which is an closed under inversion, defined in
\cite{Serra06:glt}.  Hence both sequences are also Generalized Locally
Toeplitz sequences and $\{X_N^{-1}A_N\}$ is Generalized Locally
Toeplitz with generating function $1$ and thus the singular values are
weakly clustered at the point $1$. This is enough to guarantee the
superlinear convergence of the preconditioned GMRES methods, but in
\cite{SDS:degdiff} we also show that the clustering is strong, proving
that $X_N$ is an optimal preconditioner for solving a linear system
with matrix $\{A_N\}$ with GMRES, i.e. a given error reduction is
reached within a number of iterations which is independent on the
problem size $N$.

Unfortunately there is not a fast direct solver for $X_N$, so we
resort to a Multigrid Method (MGM) with a Galerkin approach.  The MGM
\cite{Trottenberg:book} consist in constructing a solution of a linear
system by composing the action of simple iterative schemes (like
Jacobi or Gauss-Seidel), that are run on the original system and on
smaller systems derived from the first one and called {\em coarse grid
  approximations}.

\newcommand{\pro}{P}
\newcommand{\mc}{\mathcal}
\newcommand{\C}{\mathbb{C}}
\newcommand{\R}{\mathbb{R}}
\newcommand{\smo}{S}

More precisely, in order to solve a linear system $X\vec{u}=\vec{b}$
in $\R^m$, one considers a finite sequence of integers
$m_0=m>m_1>m_2>\dots>m_\ell>0$ and full-rank matrices
$\pro^{(i)}_{(i+1)}\in\R^{m_{i+1}\times m_i}$ (called {\em
  projections}) and defines the V-cycle method as
\[ 
\vec{u^{k+1}} = MGM(0,\vec{u}^k,\vec{b})
\] 
with MGM defined recursively as follows:
\begin{equation*}\label{MGMA}
  \begin{array}{c}
    \vec{u}_i^{({\rm out})}:=MGM(i,\vec{u}^{({\rm in})}_i,\vec{b}_i) \\
    \hline                                                        \\[-2.5mm]
    \begin{array}{l@{}ll|l@{}l}
      \text{If (}i=l\text{) } &
      \multicolumn{4}{@{}l}{%
        \text{Then Solve(}A_\ell\vec{u}^{({\rm  out})}_\ell=\vec{b}_\ell\text{)}}
      \\
      &\text{Else} 
      & \bf{1} & \widetilde{\vec{u}}_i   &:= \smo^\nu_i\big(\vec{u}^{({\rm in})}_i\big)\\
      && \bf{2} & \vec{r}_i &:= A_i\widetilde{\vec{u}}_i-\vec{b}_i\\
      && \bf{3} & \vec{b}_{i+1} &:= \pro^{(i)}_{(i+1)} \vec{r}_i   \\
      && \bf{4} & A_{(i+1)} &:= \pro^{(i)}_{(i+1)} A_{(i)} {(\pro^{(i)}_{(i+1)})}^\tr \\
      && \bf{5} & \vec{y}_{i+1} &:= MGM(i+1,\vec{0}_{n_{i+1}},\vec{b}_{i+1})\\
      && \bf{6} & \vec{u}^{({\rm out})}_i   &:= \widetilde{\vec{u}}_i-{(\pro^{\,i}_{i+1})}^\tr\vec{y}_{i+1} \\
    \end{array}
  \end{array}
\end{equation*}
Step 1 performs some ($\nu$) iterations of an an iterative method
(called {\em pre-smoother}) for $n_i$-dimensional linear systems that we
denoted generically as $\mc{\smo}_{i}$, chosen for its error
dampening properties, which is often taken in the Jacobi or the
Gauss-Seidel family. Then, step 2 calculates the residual of the proposed
solution and steps 3--6 define the {\em recursive coarse grid
  correction}, by projection (step 3) of the residual, sub-grid
correction (steps 4,5), and interpolation (step 6). Note that only the
smallest system (of level $\ell$) is solved exactly, while all the
others are recursively managed by reduction to low-level system and
smoothing. For more details and generalisations, see
e.g. \cite{Trottenberg:book}. 

For differential problems it is natural to construct the coarse grid
approximations with the Galerkin approach, i.e. by considering a
sequence of coarser and coarser grids with interpolation operators
$P_{(l)}^{(l+1)}$ reconstructing values of the unknown function on the
grid of level $l$ from the smaller set values on the coarser grid of
level $l+1$. In this paper it is sufficient to consider linear
(bilinear in two dimensions) interpolation operators.  This is known
to give rise to an optimal solver for a weighted laplacian. Moreover in
\cite{SDS:degdiff} we also observed that it is not necessary to bring
the MGM to convergence, but applying a single V-cycle to the GMRES
residual is enough to precondition optimally the GMRES method.

\subsection{Numerical tests}
We report here some numerical tests supporting the results of the
previous sections. For $N$
ranging from $32$ to $1024$, we integrate numerically \eqref{eq:pm}
with the Barenblatt initial data for $m=4$ from $t=0$ to $t=20/32$ in
one and two spatial dimensions, recording the number of Newton
iterations, GMRES iterations and the error against the exact solution.

\begin{figure}
  \begin{tabular}{cc}
    \includegraphics[width=0.45\linewidth]{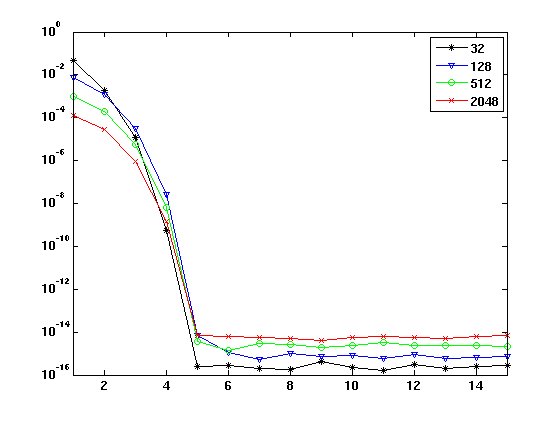}&
    \includegraphics[width=0.45\linewidth]{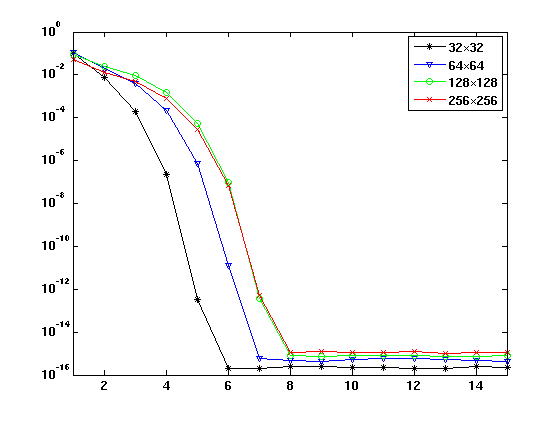}\\
    {\bf (a)} & {\bf (b)}
  \end{tabular}
  \caption{Crank-Nicholson scheme: Newton convergence history in one (a)
    and two (b) spatial dimensions}
  \label{fig:CNnewton}  
\end{figure}

In figure \ref{fig:CNnewton} we plot the quantity
$\|\vec{u}^{1,s}-\vec{u}^{1,s-1}\|$ during the Newton iterations
$s=1,2,\ldots,30$ for computing the first time step. Different symbols
and line colours correspond to different mesh sizes on the interval
(\ref{fig:CNnewton}a) and on a square domain (\ref{fig:CNnewton}b).
It is clear that the very good tolerance $10^{-6}$ is reached within
a reasonable number of iterations: $4$ in one spatial dimension and
$6$ in $\mathbb{R}^2$. This good convergence history is to a large
extent not dependent on the chosen value for $m$ (see also
\cite{SDS:degdiff}).

\begin{figure}
  \centering\includegraphics[width=0.45\linewidth]{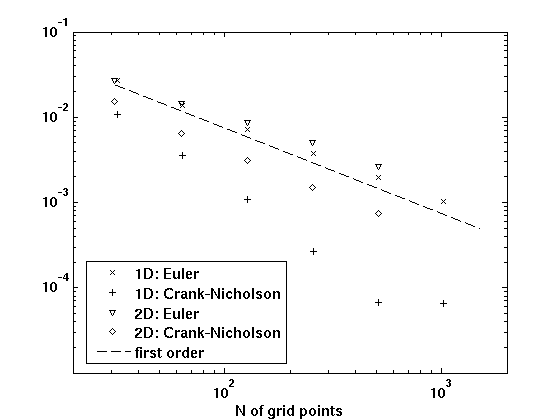}
  \caption{Error of the numerical scheme in one  and two
    dimensions, comparing with Implicit Euler}
  \label{fig:CNerr}  
\end{figure}

Using the Crank-Nicholson scheme, we observe a reduction of the error
with respect to the first order Euler scheme (see Figure
\ref{fig:CNerr}), even if the scheme does not converge with the
expected order $2$. This is due to the presence of singularities in
the exact solution. The least square fit of the errors obtained, in
one dimension, with the Crank-Nicholson scheme and represented with
plus signa in the figure gives that the error decays proportionally to
$N^{-1.5}$. In two dimensions the rate of convergence is closer to
$1$, but the errors are nevertheless lower than those obtained with
the Implicit Euler scheme.

\begin{figure}
  \begin{tabular}{cc}
    \includegraphics[width=0.45\linewidth]{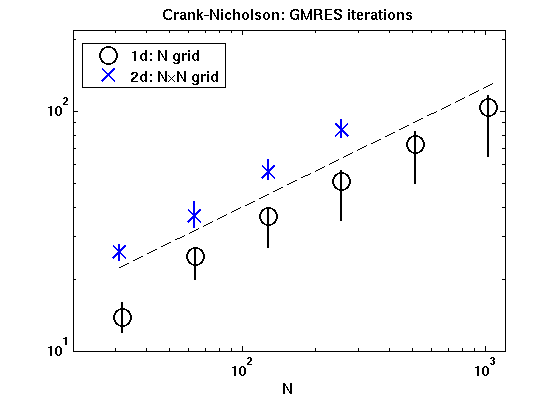}&
    \includegraphics[width=0.45\linewidth]{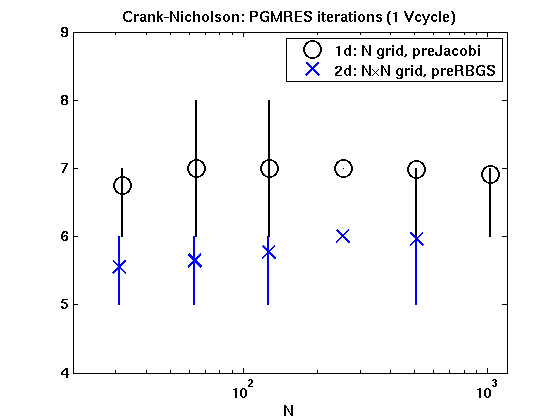}\\
    {\bf (a)} & {\bf (b)}
  \end{tabular}
  \caption{Crank-Nicholson scheme: GMRES iterations inside the Newton
    steps. (a): no preconditioning (the dashed line is the $N^{1/2}$
    slope). (b): 1 V-cycle as preconditioner. }
  \label{fig:CNPGMRES}  
\end{figure}

Finally, in Figure \ref{fig:CNPGMRES} we plot the number of iterations
of the methods for linear system.  For each value of $N$, we plot the
average number of GMRES iterations performed by the algorithm
(symbols), while the vertical lines span from the minimum to the
maximum value recorded during the integration.  Panel
\ref{fig:CNPGMRES}a compares the number of unpreconditioned GMRES
iterations on different one- and two-dimensional grids with the
$\sqrt{N}$ slope. Note that for the largest values of $N$, the
two-dimensional experiments were not possible due to memory
limitations on a PC with 8Gb of RAM. We employ MGM as preconditioner,
choosing one damped Jacobi iteration as a presmoother in the V-cycle
in the one-dimensional experiments, while for the two-dimensional we
chose the more efficient Red-Black Gauss-Seidel.  Panel
\ref{fig:CNPGMRES}b shows that applying one MGM V-cycle is optimal
(constancy of iterations number for different $N$), robust (small
variance in iterations number) and memory efficient (small number of
iterations require small amount of storage memory in GMRES).


\section{Marble sulfation}
\label{sec:sulfation} 

In this section we consider the model for the marble sulfation
problem described in \cite{ADN:sulfation}. We describe the main
features of the model only briefly, referring the reader to the
original paper for the details and more comprehensive study of the
properties of the solutions. \cite{ADN:sulfation} consider the
(simplified) chemical reaction
\begin{equation*}\label{eq:REAC}
\mathrm{CaCO_3} + \mathrm{SO_2} +\frac12\mathrm{O_2} +2\mathrm{H_2O}
\longrightarrow \mathrm{CaSO_4}\cdot2\mathrm{H_2O} + \mathrm{CO_2}.
\end{equation*}
to account for the transformation of $\mathrm{CaCO_3}$ of the marble
stone into gypsum $\mathrm{CaSO_4}\cdot2\mathrm{H_2O}$, that is
triggered in a moist atmosphere by the availability of $\mathrm{SO_2}$
at the marble surface and inside the pores of the stone.  The two main
variables of the model are $c(t,x)$ denoting the local concentration
of calcium carbonate and $s(t,x)$ the local concentration of
$\mathrm{SO_2}$. As the reaction proceeds, the calcium carbonate
concentration is reduced from the initial value $c_0$, as
$\mathrm{CaCO_3}$ is progressively replaced by gypsum. Denoting
$\varphi_0$ and $\varphi_g$ the porosity of the pristine marble and of
the gypsum, the model assumes that the porosity of the intermediate
state is well approximated by linear interpolation
\begin{equation*}\label{eq:porosity}
\varphi(c) = \varphi_g + (\varphi_0-\varphi_g)\frac{c}{c_0} =\alpha
c + \beta.
\end{equation*}
The constants $\alpha$ and $\beta$ depend on the porosity of the
material involved. The model considered in \cite{ADN:sulfation} is
described by the following system of PDEs:
\begin{equation}\label{eq:natalini}
\begin{cases}
\displaystyle
\pder{\varphi(c)s}{t} &= -\frac{a}{m_c}\varphi(c)sc +
  d\nabla\cdot\left(\varphi(c)\nabla s\right),\\[2mm]
\displaystyle
\pder{c}{t} &=  -\frac{a}{m_s}\varphi(c)sc.
\end{cases}
\end{equation}

The spatial domain $x\in\Omega$ in which \eqref{eq:natalini} is set
represents a piece of marble stone for which at least a
portion of the boundary $\partial\Omega$ is in contact with the
polluted atmosphere. In particular $\partial\Omega$ is in general
split into two parts: one represents the outer surface of
the marble sample, in contact with the air, and the complementary
part that separates the portion of the marble object of the
simulation and the rest of the monument. 

\begin{figure}
  \tikz{
    \pgfdeclarepatternformonly{gas}
    {\pgfpointorigin}
    {\pgfpoint{.25cm}{.25cm}}
    {\pgfpoint{.25cm}{.25cm}}
    {
      \draw (.035,.03) circle (.1pt);
      \draw (.1,.17) circle (.1pt);
      \draw (.15,.04) circle (.25pt);
    }
  }%
  \begin{center}
    \begin{tabular}{cc}
      \begin{tikzpicture}[scale=2]
        \fill[pattern=gas,pattern color=gray] (0,0) rectangle (-1,2);
        \fill[pattern=bricks,pattern color=black!20!white] (0,0) rectangle (2,2);
        \draw[yshift=1cm,->] (-.5,0) -- (1.5,0) node[pos=0.5,above]{$\mathbf{\Omega}$} node[below]{$x$};
        \draw[yshift=1cm,line width=2pt] (0,3pt) -- (0,-3pt) node[below]{$0$};
        \draw[yshift=1cm,line width=2pt, dotted] (1,3pt) -- (1,-3pt) node[below]{$1$};
      \end{tikzpicture}
      &
      \begin{tikzpicture}[scale=2]
        \fill[pattern=gas,pattern color=gray] (-.5,1.5) rectangle (0,-.5) rectangle (1.5,0);
        
        \fill[pattern=bricks,pattern color=black!20!white] (0,0) rectangle (1.5,1.5);
        \node at (.5,.5) {$\mathbf{\Omega}$};
        \draw[line width=2pt] (0,1) -- (0,0) -- (1,0);
        \draw[dotted,line width=2pt] (1,0) -- (1,1) -- (0,1);
      \end{tikzpicture}
      \\
      {\bf(a)}&{\bf(b)}
    \end{tabular}
  \end{center}
  \caption{Sample domains $\Omega$ for problem \eqref{eq:natalini} are
    shown for the 1D setting (a) and 2D setting (b). The ``brick
    pattern'' area represents the marble stone, while the dotted area
    is air.  The boundary is drawn with a solid line where Dirichlet
    boundary conditions are applied and with a dotted line where
    free-flow boundary conditions are imposed.}
  \label{fig:domain}
\end{figure}
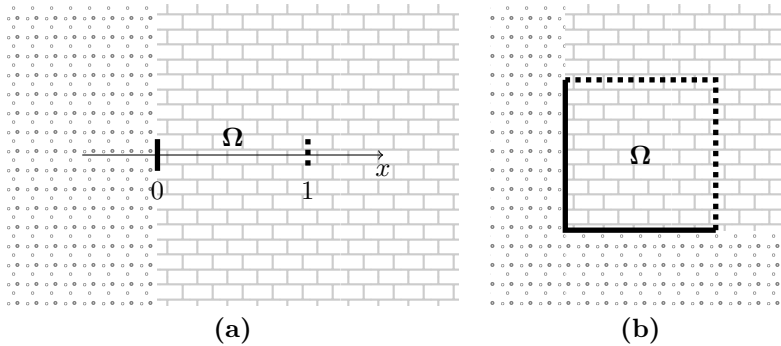

Examples of one and two-dimensional such domains are shown in Figure
\ref{fig:domain}. In order to study the formation of a gypsum crust on
a flat area of a monument, it is sufficient to take a 1-dimensional
domain like in the left panel of the figure. In order to study a
corner-like feature of a work of art, like the edge of a monument or a
long decoration in relief, one would employ a 2-dimensional domain as
the one in the right panel of the figure. Obviously more complex
shapes need 3-dimensional domains representing faithfully the volume
occupied by the marble.

Boundary conditions are set by imposing the value of $s$ on the outer
boundary and by imposing free-flow conditions for $s$ on the inner
boundary. In particular, as an example of a one-dimensional setting,
we take $x\in\Omega=[0,1]$ where $x=0$ corresponds to the outer
boundary of the marble stone, in contact with the polluted air, and
$x=1$ the inner side. The boundary conditions are illustrated in
Figure \ref{fig:domain}: they are of Dirichlet type, imposing $s(0,t)$
on the outer boundary and of free-flow type $\pder{s}{x}(1,t)=0$ on
the inner side.

The parameters $m_s$ and $m_c$ are fixed by the physical properties
of the species involved in the reaction and make sure that the mass
balance is fulfilled. On the other hand $a$ represents the reaction
rate and it depends (among other things) on the moisture of the air
and on the temperature. \cite{ADN:sulfation} describes its central
role in the analysis of the solutions of the model equations. In
particular, if $u(t,x)$ is a solution of \eqref{eq:natalini} for a
given value of $a$, then $\tilde{u}(t,x)=u(t/a^2,x/a)$ is a solution
of \eqref{eq:natalini} for $a=1$. This observation on one hand plays
a fundamental role in establishing the long-time asymptotics of the
solution and would allow to perform the simulations with $a=1$ and
then rescale the numerical solutions appropriately to take the
reaction rate into
account. However, because of its role as
fundamental physical parameter of the model, here we prefer to keep
$a$ explicitly into the equations and perform simulations seeking
numerical solution of the model in the form \eqref{eq:natalini}.
Moreover this is beneficial in view of the more complete model
including a non-constant $a$ has been described in \cite{AFNT07:model}.


We consider, as in \cite{ADN:sulfation} $\alpha=0.01$, $\beta=0.1$,
$d=1$, $m_s=64.06$, $m_c=100.09$ and use $h=1/64$ or $h=1/128$ while
varying $a$ from $1$ to $10^5$.

\subsection{Discretization}
\label{ssec:model:discr}
As in the scalar case, we consider first and second order implicit
time discretizations, namely the Implicit Euler scheme
(Crandall-Liggett formula) and the Crank-Nicholson scheme.  The
general setup for the scheme is the same as in the scalar case: we
discretize the spatial domain and the elliptic differential operator
with finite differences, write the time-advancement problem as an
implicit equation and set up a Newton scheme to solve it. This
procedure is advantageous if the Newton scheme converges in a
reasonable number of iterations and one can devise an optimal
preconditioner for the linear system that has to be solved at each
Newton step.

For the space discretization, we
denote $x_{\xi}=0+\xi h\in\Omega$. Approximating the elliptic operator
along the same lines as in \eqref{eq:FDlaplacian}, we consider the
second order finite difference formula
\begin{eqnarray}\label{}
\left.\partial_x(\varphi(c) \partial_x s)\right|_{x_j}
 & = & \frac{\varphi(c(x_{j+1/2}))(s(x_{j+1})-s(x_j))}{h^2} -\\
  & & -\frac{\varphi(c(x_{j-1/2}))(s(x_{j})-s(x_{j-1}))}{h^2}.
\end{eqnarray}
This in turn suggests that we employ two staggered grids in the domain
$\Omega$: the grid $x_j$ ($j\in\mathbb{N}$) with the unknowns $s^n_j$
for $s(t^n,x_j)$ and the grid $x_{j+1/2}$ ($j\in\mathbb{N}$) with the
unknowns $c^n_{j+1/2}$ for $c(t^n,x_{j+1/2})$. For short, we also
denote $\varphi^n_{j+1/2}=\varphi(c^n_{j+1/2})$.  For ease of
reference, we will denote the two grid also by ``integer grid'' and
``half-integer grid''.

We write explicitly the formulas for the Crank-Nicholson time
discretization \eqref{eq:CN}, pointing out that the case of
Crandall-Liggett (Implicit Euler) scheme \eqref{eq:EI} can be
similarly dealt with. Thus we consider the scheme that computes
$s^{n}_j$ and $c^{n}_{j+1/2}$ solving the nonlinear system of
equations:
\begin{equation}
  \label{eq:NATF}
  \begin{cases}
    0=\mathbf{F}^{(s)}(\vec{s}^n,\vec{c}^n) =
    & \begin{aligned}[t]
      &\Phi^n \vec{s}^n 
      +\frac{\dt}2 \frac{a}{m_c} C^n \vec{s}^n
      +\frac{\dt}2 d L_{\varphi^n}\vec{s}^n\\ 
      &- \Phi^{n-1} \vec{s}^{n-1}
      +\frac{\dt}2 \frac{a}{m_c} C^{n-1} \vec{s}^{n-1}
      +\frac{\dt}2 d L_{\varphi^{n-1}}\vec{s}^{n-1}
    \end{aligned}
    \\
    0=\mathbf{F}^{(c)}(\vec{s}^n,\vec{c}^n) =
    & \vec{c}^n - \vec{c}^{n-1}
    + \frac{\dt}2 \frac{a}{m_s} S^n \vec{c}^n
    + \frac{\dt}2 \frac{a}{m_s} S^{n-1} \vec{c}^{n-1}
  \end{cases}
\end{equation}
where $\vec{s}^n=[s^n_1,s^n_2,\ldots]^\tr$,
$\vec{c}^n=[c^n_{1/2},c^n_{3/2},\ldots]^\tr$ and
\begin{subequations}
\begin{gather}
\Phi^n = \diag_k\left[\frac{\varphi^n_{k+1/2}+\varphi^n_{k-1/2}}{2}\right]\\
C^n = \diag_k\left[ 
  \frac{\varphi^n_{k+1/2}c^n_{k+1/2}+\varphi^n_{k-1/2}c^n_{k-1/2}}{2}
\right]\\
L_{\varphi^n} = \tridiag_k\left[-\varphi^n_{k-1/2},
  \varphi^n_{k-1/2}+\varphi^n_{k+1/2},
  -\varphi^n_{k+1/2}
\right]\\
S=\diag_k\left[
  \varphi^n_{k+1/2}\left(\frac{s^n_{k+1}+s^n_k}2\right)
\right]
\end{gather}
\end{subequations}
Note that equations \eqref{eq:NATF} are not linear, since also the
matrices $C$, $S$, $L_\phi$ and $\Phi$ depend on $\vec{c}^n$ and
$\vec{s}^n$, either directly or via the (linear) function $\varphi$. It
is important to note that the matrix $L_\varphi$ is defined as $L_D$
of \eqref{eq:Ld}, but it depends only on the half of the unknowns
of the problem: precisely $L_\varphi$ depends on 
$\vec{c}$ and it multiplies $\vec{s}$ in formula
\eqref{eq:NATF}.

The two staggered grids represent a sort of finite difference analogue
of the approximation with $P1$ (for $s(x)$) and $P0$ (for $c(x)$)
conforming finite elements considered in \cite{ADN:sulfation}. The
results obtained here on preconditioning should also be applicable
with little modifications in that case too.

Boundary conditions are imposed considering $j=1,2,\ldots,N$ in
\eqref{eq:NATF} and assuming at all time steps a given value for $s_0$
(Dirichlet boundary condition at $x=0$) and that $s_{N+1}=s_{N-1}$
(homogeneous Neumann boundary condition at $x=1$).  Thus the
expressions of $\mathbf{F}^{(s)}_1$ and $\mathbf{F}^{(s)}_N$ are
modified accordingly with respect to those in \eqref{eq:NATF},
together with the correspondent elements in the Jacobian
\eqref{eq:NATJ}. The $2N$ unknowns are collected in a vector $\vec{u}$
with the ordering
$\vec{u}=[s_1,s_2,\ldots,s_N,c_{1/2},\ldots,c_{N-1/2}]^\tr$. The
corresponding sparsity structure of the Jacobian matrix is illustrated
in Figure \ref{fig:sparsity}. For the actual implementation it is
easier to define the ``porous concentration'' and set the Dirichlet
boundary condition as
$\left.\varphi_{j-1/2}s_j\right|_{j=0}=\rho_{s_0}=1$.

Both the Crank-Nicholson and the Crandall-Liggett formulas give rise
to an unconditionally stable scheme.  Following the results previously
established, in order to solve the nonlinear problem \eqref{eq:NATF}
we set up Newton iterations. To this end we need the Jacobian matrix,
which is naturally split into four $N \times N$ block as
\begin{equation*}
J=F^\prime =
\left[\begin{array}{c|c}
J^{s}_s & J^{s}_c \\\hline J^{c}_s & J^{c}_c
\end{array}
\right]
\qquad \qquad \vec{u}=
\begin{pmatrix}
\vec{u}_s \\\vec{u}_c
\end{pmatrix}
\end{equation*}
The entries (disregarding boundary conditions) are:
\begin{equation}\label{eq:NATJ}
\begin{aligned}
  \left[J^s_s\right]_{j,k}=
  \pder{\mathbf{F}^{(s)}_j}{s_k}
  =&\frac{\varphi_{j+1/2}+\varphi_{j-1/2}}{2}\delta_{jk}
  +\frac{\dt}2 \frac{a}{m_c}
   \frac{\varphi_{j+1/2}c_{j+1/2}+\varphi_{j-1/2}c_{j-1/2}}{2}\delta_{jk}
  \\
  &+\frac{d}2\frac{\dt}{h^2} \left[ -\varphi_{j-1/2}\delta_{k,j-1}
    +(\varphi_{j-1/2}+\varphi_{j+1/2})\delta_{k,j}
    -\varphi_{j+1/2}\delta_{k,j+1}
  \right],
  \\[1mm]
  \left[J^s_c\right]_{j,k}=
  \pder{\mathbf{F}^{(s)}_j}{c_{k+1/2}}
  =&\frac{\varphi^{\prime}_{j+1/2}s_j\delta_{jk}+\varphi^{\prime}_{j-1/2}s_j\delta_{j,k+1}}{2}\\
  &+\frac{\dt}2 \frac{a}{m_c}
  \frac{(\varphi^{\prime}_{j+1/2}c_{j+1/2}+\varphi_{j+1/2})\delta_{jk}
    +(\varphi^{\prime}_{j-1/2}c_{j-1/2}+\varphi_{j-1/2})\delta_{j,k+1}
  }{2}
  \\
  &+\frac{d}2\frac{\dt}{h^2}
  \left[ \varphi^{\prime}_{j-1/2}(s_j-s_{j-1})\delta_{j,k+1}
    -\varphi^{\prime}_{j+1/2}(s_{j+1}-s_{j})\delta_{j,k}
  \right],
  \\
  \left[J^c_s\right]_{j,k}=
  \pder{\mathbf{F}^{(c)}_{j+1/2}}{s_k}
  =&\frac{\dt}2\frac{a}{m_s}
  \varphi_{j+1/2}c_{j+1/2}\frac{\delta_{j,k-1}+\delta_{jk}}{2},
  \\
  \left[J^c_c\right]_{j,k}=
  \pder{\mathbf{F}^{(c)}_{j+1/2}}{c_{k+1/2}}
  =&\left[1
    +\frac{\dt}2 \frac{a}{m_s}(\varphi^{\prime}_{j+1/2}c_{j+1/2}+\varphi_{j+1/2})
    \frac{s_{j+1}+s_j}{2}
    \right]
    \delta_{jk}.
\end{aligned}
\end{equation}

\begin{figure}
\hfil
\includegraphics[width=0.45\textwidth]{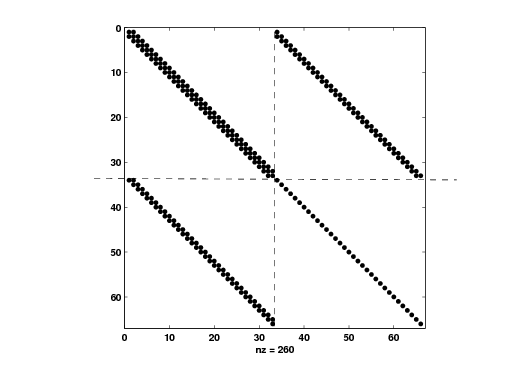}
\hfil
\caption{Sparsity structure of the Jacobian matrix \eqref{eq:NATJ}.}
\label{fig:sparsity}
\end{figure}

The sparsity structure of the Jacobian matrix with entries defined in
\eqref{eq:NATJ} is shown in Figure \ref{fig:sparsity}. A more detailed
analysis of the matrix will be carried out in the next section.

\subsection{Solving the linear system}
At each Newton iterations, we have to solve a linear system with
matrix $J$, which is not symmetric and thus we employ GMRES as the
main Krylov solver.

\begin{figure}
\centerline{\includegraphics[width=0.45\textwidth]{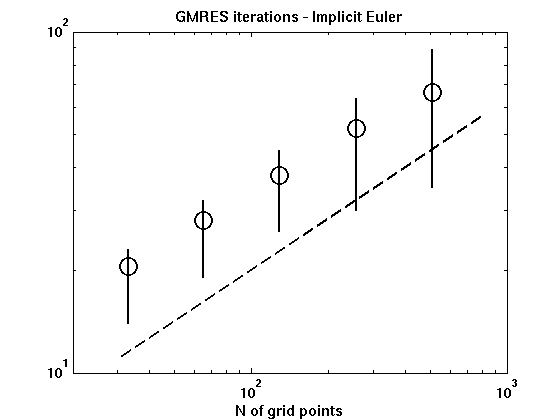}}
\caption{Number (average, min and max) of GMRES iterates per
  timestep with $A=1$, for 3 values of $N$. No preconditioner was used
  in this test. The dashed line indicates the $N^{1/2}$ slope.}
\label{fig:NAT:GMRESnoP}
\end{figure}

We observe that the top-left $J^s_s$ block is given by
\[ J^s_s = \Phi+\frac{\dt}2
\frac{a}{m_c}C+\frac12\frac{\dt}{h^2}dL_\varphi
\]
which is very similar to \eqref{Jf2}, except that the identity is
replaced by the diagonal matrix $\Phi$ with $O(1)$ entries and the
tridiagonal $Y$ term of \eqref{Jf2} is not present, corresponding
instead to the third term of the \(J^s_c\) block. The extra term of
$J^s_s$, involving the diagonal matrix $C$ has entries of order $\dt$.
Hence we expect $J_s^s$ to be spectrally not too different from
$F^\prime$ of \eqref{Jf2} and thus unpreconditioned GMRES iterations
count to grow as $\sqrt{N}$, which is indeed confirmed in Figure
\ref{fig:NAT:GMRESnoP}, where we plot the average, minimum and maximum
number of GMRES iterations needed in the case $a=1$ and for different
values of the number of grid points $N$. A least square fit gives
$N^{0.5217}$ for the number of iterations.

In order to devise a preconditioning strategy, given the structure of
$J^s_s$, we can employ a V-cycle on this block, if we can deal
optimally with the rest of the matrix. To this end we observe that the
lower left block \(J^c_s\) has nonzero entries only on two diagonals
and these decay as $O(\dt)$, while the bottom right block \(J^c_c\) is
the identity matrix plus a diagonal matrix with $O(\dt)$ entries.

\begin{theorem}\label{teo:Pnat}
  The upper triangular part of $J$,
  \begin{equation}\label{eq:Pnat}
    P= \left[\begin{array}{c|c}
        J^{s}_s & J^{s}_c \\\hline {\bf0} & J^{c}_c
       \end{array}
     \right]
  \end{equation}
  is an optimal preconditioner for $J$, assuming that the function
  $\varphi(c)$ is bounded away from $0$, or equivalently that
  $\beta>0$.
\end{theorem}
\begin{proof}
First observe that the diagonal blocks are nonsingular, so that
\[
P^{-1} = 
\left[\begin{array}{c|c}
    (J^s_s)^{-1}&-(J^s_s)^{-1}(J^s_c)(J^c_c)^{-1}\\\hline{\bf0}&(J^c_c)^{-1}
\end{array}
\right]
\]
and the preconditioned system has matrix
\[ 
P^{-1}J =
 \left[\begin{array}{c|c}
{\bf 1} -\left(J^s_s\right)^{-1}\left(J^s_c\right)\left(J^c_c\right)^{-1}\left(J^c_s\right)
 & {\bf 0} \\
\hline \left(J^c_c\right)^{-1} \left(J^{c}_s\right) & 
{\bf 1}
\end{array}
\right]
\]
where ${\bf0}$ denotes the null matrix and ${\bf1}$ the identity
matrix.

We now show that all entries of $P^{-1}J$ are negligible except the
diagonal ones. In fact $J^c_c$ is diagonal with entries equal to
$1+O(\dt)$ and thus its inverse has the same property. Since $J^c_s$
is tridiagonal with entries of $O(a\dt)$, the same is true for the
lower-left block of $P^{-1}J$. Gershgorin circles arising from the
lower half of the matrix are thus centred at $1$ in the complex plane
and have radii decaying as $O(\dt)$.

We now turn to consider the upper half of the matrix, where we observe
that
$\|\left(J^s_c\right)\left(J^c_c\right)^{-1}\left(J^c_s\right)\|_\infty=O(\dt)$
and thus it suffices to show that $\|\left(J^s_s\right)^{-1}\|_\infty$
is bounded to conclude that the Gershgorin circles arising from the
upper half of the matrix are centred at $1+O(\dt)$ in the complex
plane and have radii decaying as $O(\dt)$.

To this end, split
\[ J^s_s = Z -W  = Z({\bf1}-Z^{-1}W)\]
where $Z$ is the diagonal part, which is
\[ Z = \diag_k(z_k), \qquad
z_k=\varphi_k+\dt\frac{a}{m_c}\widetilde{\varphi}_k
+2\frac{\dt}{h^2}\varphi_k 
\] 
where $\varphi_k=\tfrac12(\varphi_{k+1/2}+\varphi_{k-1/2})$ and
$\widetilde{\varphi}_k=\tfrac12(\varphi_{k+1/2}c_{k+1/2}+\varphi_{k-1/2}c_{k-1/2})$.
Since $Z$ is diagonal, we easily get the estimate
\[
\|Z^{-1}\|_\infty 
\leq \max_k \frac{1}{z_k}
=\frac{h^2}{\dt} \max_k
\frac{1}{2\varphi_k}\left(1+O\left(\tfrac{h^2}{\dt}\right)\right)
=O\left(\tfrac{h^2}{\dt}\right)
\]

Next observe that
$Z^{-1}W=\tfrac{\dt}{h^2}\tridiag_k(\tfrac{1}{z_k}[\varphi_{k-1/2},0,\varphi_{k+1/2}])$
and thus 
\[ \left\|Z^{-1}W\right\|_\infty 
= \left\|
  \tridiag_k\left(
    \frac{[\varphi_{k-1/2},0,\varphi_{k+1/2}]}%
    {2\varphi_k
      +\tfrac{h^2}{\dt}\varphi_k
      +h^2\tfrac{a}{m_c}\widetilde{\varphi}_k
    }
  \right)
\right\|_\infty 
\leq \max_k \frac{2\varphi_k}{2\varphi_k
  +\tfrac{h^2}{\dt}\varphi_k
  +h^2\tfrac{a}{m_c}\widetilde{\varphi}_k
}
\leq 1-Ch
\]
for some small positive constant $C$.
Therefore 
\[
\left\|\left({\bf1}-Z^{-1}W\right)^{-1}\right\|_\infty
\leq \sum_{j=0}^\infty \left\|Z^{-1}W\right\|_\infty^j
\leq \frac{1}{Ch} 
\]
and
\[ 
\|(J^s_s)^{-1}\|_\infty 
= \left\|\left(Z\left({\bf1}-Z^{-1}W\right)\right)^{-1}\right\|_\infty
=\left\|\left({\bf1}-Z^{-1}W\right)^{-1}\right\|_\infty
\left\| Z^{-1}\right\|_\infty
=O(1)
\]
\end{proof}

\begin{remark}
  When applying the preconditioner, the block triangular system
  $P\vec{y}=\vec{b}$ is solved as
  \[ \vec{y}_c = (J^c_c)^{-1}\vec{b}_c \qquad \vec{y}_s =
  (J^s_s)^{-1}(\vec{b}_s-J^{s}_c\vec{y}_c) 
  \]
  where the deponents $s$ and $c$ refer to the upper and, respectively,
  the lower half of the vectors.
  The previous result shows that the spectrum of $P^{-1}J$ is strongly
  clustered at $1$ independently on the discretization parameter $h$
  and we expect the block-preconditioner $P$ to be optimal.  For the
  whole preconditioner to be optimal, however we need an optimal
  solver for the $J^s_s$ block. However $J^s_s$ is the sum of two
  diagonal matrices and a tridiagonal matrix which is the
  discretization of a laplacian operator, regularised with the
  (strictly positive) function $\varphi(c(x))$, and thus has spectral
  properties close to those of $X_N$ studied in Section
  \ref{ssec:alglin}. As in the scalar case, a MGM (e.g. with 1 damped
  Jacobi as presmoother and a Galerkin approach with linear
  interpolation) is an optimal solver for this block.
\end{remark}

\begin{figure}
\begin{tabular}{cc}
\includegraphics[width=0.45\textwidth]{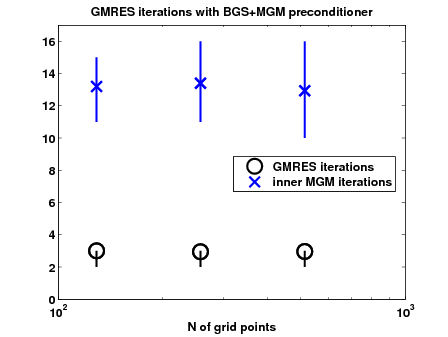}
&
\includegraphics[width=0.45\textwidth]{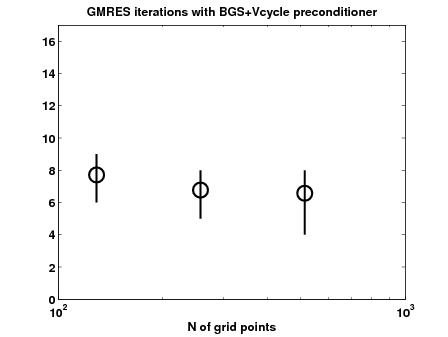}
\\
{\bf(a)} & {\bf(b)}
\end{tabular}  
\caption{ Number (average, min and max) of GMRES iterates per timestep
  with $A=1$, for 3 values of $N$. (a) Block Gauss-Seidel and MGM for
  the upper left block was used as a preconditioner. Blue symbols and
  lines refer to the number of inner MGM iterations. (b) GMRES
  iterations when performing only 1 V-cycle of the MGM.}
\label{fig:NAT:GMRESMGM}
\end{figure}

Our preconditioner is thus the Gauss-Seidel preconditioner at block
level, with MGM on the $(s,s)$ block. The $(c,c)$ block does not need
an inner preconditioner since it is diagonal and can be solved
directly.  This strategy yields an optimal preconditioner,
i.e. renders the number of GMRES iterations independent from $N$, as
confirmed by the numerical experiments shown in Figure
\ref{fig:NAT:GMRESMGM}). In the panel \ref{fig:NAT:GMRESMGM}a we
employ MGM driven to convergence as solver for the $(s,s)$ block in
the preconditioner: GMRES converges in 2-3 iterations, requiring 10-15
MGM cycles at each iteration. In panel panel \ref{fig:NAT:GMRESMGM}a
we employ only a single MGM V-cycle as inner preconditioner for the
$(s,s)$ block: the number of GMRES iterations grows slightly (6--8),
but this procedure is overall more efficient.  We observe an
impressive series of good features: minimal average computational
cost, minimal number of iterations, minimal variance in the latter
number meaning a strong robustness of the procedure.

\subsection{Simulations and performance of the algorithm}
\begin{figure}
\begin{tabular}{cc}
  \includegraphics[width=0.45\textwidth]{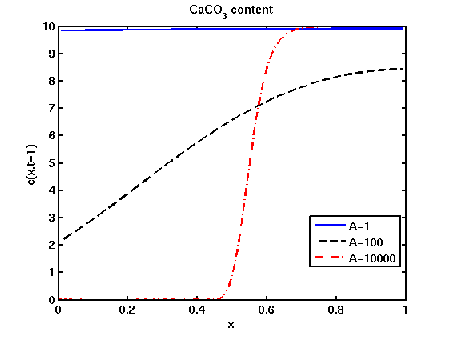}
  &
  \includegraphics[width=0.45\textwidth]{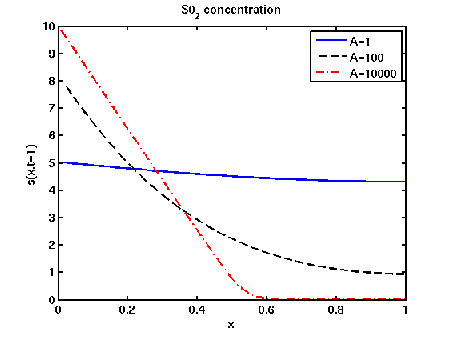}
  \\
  {\bf(a)} & {\bf(b)}
\end{tabular}
\caption{For different values of $a$, marble content and
  $\mathrm{SO}_2$ concentration inside the stone predicted by the
  model \eqref{eq:natalini} for $t=1$.}
\label{fig:natAvario}
\end{figure}

\begin{figure}
  \begin{tabular}{cc}
    \includegraphics[width=0.45\textwidth]{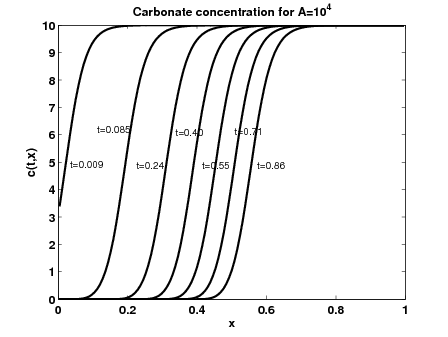}
    &
    \includegraphics[width=0.45\textwidth]{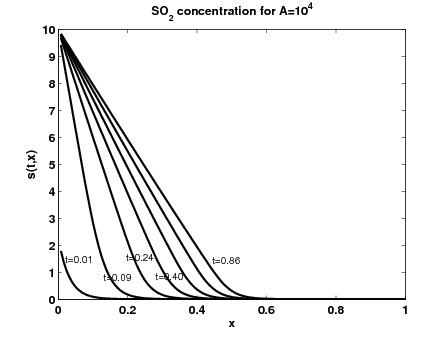}
    \\
    {\bf(a)} & {\bf(b)}
  \end{tabular}
  \caption{Temporal evolution of the calcium carbonate and sulfate
    concentration predicted by the model \eqref{eq:natalini} with
    $a=10^4$.}
\label{fig:film}
\end{figure}

In Figure \ref{fig:natAvario} we plot some typical curves obtained
from the simulations with the model \eqref{eq:natalini}. Note that
for bigger values $a$, the reaction is faster and a boundary layer
appears. For $a=10^4$, Figure \ref{fig:film} shows the temporal
evolution of the two main variables: while $\mathrm{SO_2}$
penetrates deeper and deeper into the stone (b), calcium
carbonates is substituted by the more porous gypsum in a narrow
spatial band where the curve $c(t,x)$ presents a boundary layer.
Once formed, this transition region travels towards the interior of
the stone (a). The self-similarity of the solutions of
\eqref{eq:natalini} under rescaling of the temporal and spatial
variables mentioned at the beginning of Section \ref{sec:sulfation}
implies that the boundary layer observed for $a=10^4$ will also
appear for lower values of $a$, if the solutions were sought for a
larger temporal and spatial domain.

\paragraph{Newton iterations}
In Figure \ref{fig:solfNewton} we plot the average, minimum and
maximum number of Newton iterations used by the numerical method to
solve the nonlinear equation \eqref{eq:NATF} at each timestep.

\begin{figure}
\hfil
\includegraphics[width=0.45\textwidth]{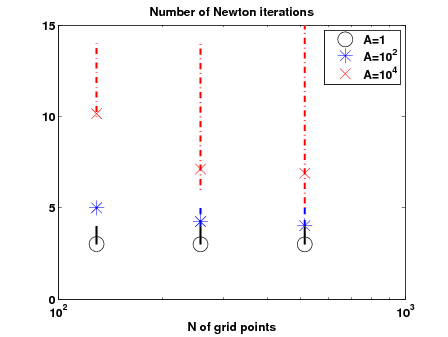}
\hfil 
\caption{Number (average, min and max) of Newton iterates per
  timestep.}
\label{fig:solfNewton}
\end{figure}

For $a=1$ (black circles) we note that the number of iterations is
almost constant when the number of grid points is increased.
Furthermore, for a given number of points employed in the
discretization, the number of Newton iterations increases very
moderately even when $a$ is increased by several orders of magnitude:
e.g. for $N=128$ we need an average of $3$ Newton iterations per
timestep for $a=1$, $5$ for $a=100$ (blue stars), and $10$ for
$a=10000$ (red crosses). Finally we point out that for higher values
of $N$, the number of Newton iterations decreases slightly since the
grid becomes able to resolve better the boundary layer.

\subsection{Asymptotics for the front position}
The asymptotic analysis of \cite{GN07} predicts that the front of the
travelling wave of $c(t,x)$ that separates the gypsum dominated phase
from the carbonate dominated phase and that moves inwards in the
marble sample asymptotically behaves as $x_{\text{front}}\sim\sqrt{t}$.

\begin{figure}
  \begin{tabular}{cc}
    \includegraphics[width=0.45\textwidth]{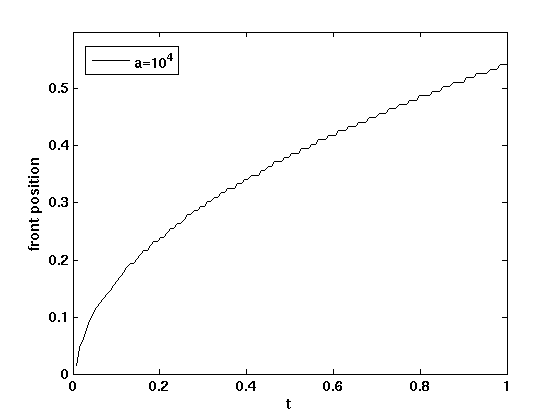}
    &
    \includegraphics[width=0.45\textwidth]{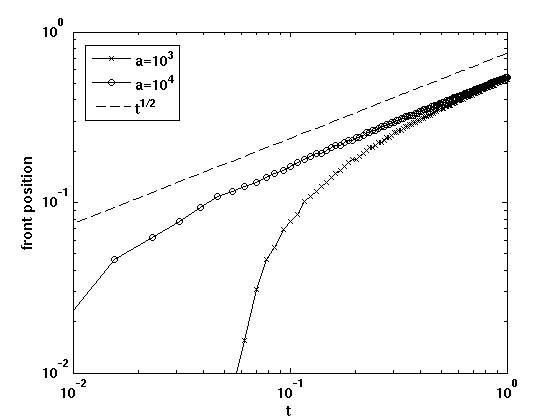}
    \\
    {\bf(a)}&{\bf(b)}
  \end{tabular}
  \caption{Gypsum-carbonate front position in a sulfation problem:
  numerical simulations (solid lines) and predicted
  asymptotics (dashed line) in b}
\label{fig:experiment}
\end{figure}

We performed numerical experiments to test this prediction and to
check how fast the front approaches this asymptotics.
In order to perform the comparison, we extracted the information on
the front position from the numerical solutions $c^n_{j+1/2}$ by
identifying the gypsum-carbonate front with the point with steepest
gradient of $c(t^n,x)$. 

For $a=10^4$, Figure \ref{fig:experiment}a shows the position of the
front.  Note that the step-like behaviour of the numerical front that
is apparent in some regions of the graph is due to the finite spatial
resolution of the simulation ($h=1/128$).  In order to check the
asymptotics, we plot the position also in double logarithmic scale in
Figure \ref{fig:experiment}b, together with the $\sqrt{t}$ slope
(dashed line). We note that both simulations agree with the slope of
the asymptotics and that for the smaller value of $a$, the solution
approaches the asymptotics more slowly.

\subsection{Sample application in 2D}
\begin{figure}
\begin{center}
\hfil
\includegraphics[width=0.8\textwidth]{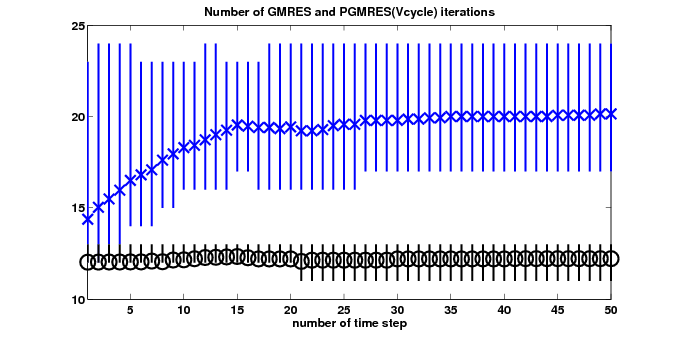}
\hfil
\end{center}
\caption{Number of GMRES iterations per Newton step: for each timestep
we plot the average, minimum and maximum number of linear
iterations. Without preconditioner: blue, crosses. With V-cycle
preconditioner: black, circles.}
\label{fig:SOLF2DPGMRES}
\end{figure}

In this section we present a numerical simulation of equations
\eqref{eq:natalini} in the two dimensional setting of figure
\ref{fig:domain}. We consider again two staggered quadrangular regular
grids in $\Omega=[0,1]\times[0,1]$. When using $N$ points per
direction, we denote $x_\xi=\xi/N$ and $y_\xi=\xi/N$. We generalize
the construction of Section \ref{ssec:model:discr} considering two
staggered grids: the {\em integer grid} is the set of points
\(\{(x_i,y_j)\}_{i,j=0}^N\) carrying the values $s_{i,j}$ of the
$\mathrm{SO_2}$ concentration field $s(x,y)$ and the {\em half-integer
  grid} is the set \(\{(x_{i+1/2},y_{j+1/2})\}_{i,j=1}^N\) carrying
the values $c_{i+1/2,j+1/2}$ of the calcium carbonate concentration
field. The discretization of the elliptic operator is then generalized
in the usual way to the two dimensional setting, and the new form of
the fixed point problem \eqref{eq:NATF} and its laplacian
\eqref{eq:NATJ} are derived. The numerical scheme now requires, at
each timestep, the solution of a system of $2N^2$ nonlinear
equations. Using again the Newton method, at each iteration we need to
solve a sparse linear system with a matrix of dimension
$2N^2\times2N^2$.

The Jacobian matrix has the same block structure described in the
one-dimensional case (see also Figure \ref{fig:sparsity}).  Since it
is not symmetric, we use GMRES as main Krylov solver with specialised
structured preconditioners as in the one dimensional case.  The main
difference in fact is that now the $J^s_s$ block is a weighted
two-dimensional laplacian plus diagonal corrections that are small (in
the sense of order of $h$). The $J^c_c$ block remains diagonal with
elements equal to $1+O(\dt)$ and the preconditioner \eqref{eq:Pnat}
can be applied. The off-diagonal blocks have more non-zero diagonals,
but their elements are still small and Theorem \ref{teo:Pnat} can be
generalized to the two-dimensional setting.

Here we consider only the best preconditioner of those evaluated in
the one dimensional setting, namely the upper triangular part of the
Jacobian matrix, where we perform only 1 V-cycle on the $(s,s)$ block.
In Figure \ref{fig:SOLF2DPGMRES} we study the effectiveness of this
preconditioning technique. On a $32\times32$ grid, we observe that
unpreconditioned GMRES requires an average of $15$ to $20$ iterations
to solve the Jacobian linear system in each Newton step, with frequent
peaks of $24$ iterations (blue crosses in the figure are the average
values, blue lines the minimum to maximum range). Moreover the number
of iterations is not constant but depends on the timestep. On the
contrary the preconditioned method employs always an average of $12$
PGMRES iterations, with little variability both within the time step
and across the different times (black circles and lines).

\begin{figure}
\begin{center}
\hfil
\includegraphics[width=0.8\textwidth]{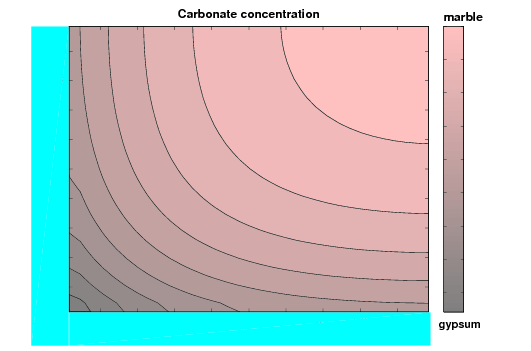}
\hfil
\end{center}
\caption{Simulation of marble sulfation in two dimensions.}
\label{fig:SOLF2D}
\end{figure}

In Figure \ref{fig:SOLF2D} we plot the solution obtained for
$a=10$. We recall the the marble is in contact with the polluted air at
the bottom and left boundary (cyan regions), while at the top and
right boundary we apply free flow conditions. Both the colour code and
the isolines refer to the carbonate concentration in the stone. We
observe a clear deformation of the $\mathrm{CaCO_3}$ field near the
corner, clearly indicating that $\mathrm{SO_2}$, penetrating from both
sides, causes an enhanced loss of material: if the gypsum crust were to
fall off here, the sharp edge would be chipped off and 
the shape of the stone would be permanently changed.
This simulation, although performed at low resolution ($32\times32$
grids) and with a moderate value of $a$, already indicates the
relevance of our project of developing accurate numerical simulators
for realistic geometries of the domain $\Omega$ in two and three
dimensions.

\section{Conclusions and future developments}
\label{sec:conclusions}

The novel contribution of this paper relied in the proposal of a fully
implicit numerical method for dealing with the nonlinear PDE, in its
convergence and stability analysis, and in the study of the related
computational cost. Indeed the nonlinear nature of the underlying
mathematical model required the application of a fixed point scheme.
We have identified the classical Newton method in which, at every
step, the solution of a large, locally structured, linear system has
been handled by using specialised iterative or multi-iterative
solvers. In particular, the spectral analysis of the relevant matrices
has been crucial for identifying appropriate preconditioned Krylov
methods with efficient V-cycle preconditioners.  Numerical experiments
for the validation of our analysis complement this contribution, which
is aimed to provide a non-invasive tool for a quantitative forecast of
the damage evolution in a given monument.

In particular we considered the application of the above-mentioned
techniques to the numerical approximation of a mathematical model
describing the damage of marble monuments by the sulfation
process. The use of our resulting fast integration algorithms allows
to exploit the model and its predictive power for the strategy known
as {\em planned conservation}, that is the novel approach that
privileges the study and prevention of the damages to delay and
optimise the actual restoration works. We showed both one dimensional
and two dimensional numerical simulations, using simple domains.

For future work, two main directions appear naturally. The first is to
finite element methods for the space discretization in order to deal
with more realistic domains in 2D and 3D, that can model a real
architectural item with a complicate geometry. In this setting we
expect algebraic (linear) MGM or even nonlinear multigrid (FAS) to
play an important role.  The other natural extension of the numerical
treatment for the sulfation problem involves considering the
3-equations model in \cite{AFNT07:model} and/or a model for a
remediation technique, like the one in \cite{CGNNS:teos} (this will
include consolidation models that is systems of the type
$u_t=(D(u)(p(u)_x))_x$). As a long term goal, being able to simulate
both the damage and the remediation process with validated
mathematical models and numerical methods would allow to perform
numerical experiments of restoration works.

\bibliographystyle{alpha}
\bibliography{MonumentsCL}

\newcommand{\etalchar}[1]{$^{#1}$}
\begin{thebibliography}{CGN{\etalchar{+}}09}

\bibitem[ADDN04]{ADN:sulfation}
D.~Aregba~Driollet, F.~Diele, and R.~Natalini.
\newblock A mathematical model for the {$\mathrm{SO}_2$} aggression to calcium
  carbonate stones: numerical approximation and asymptotic analysis.
\newblock {\em SIAM J. Appl. Math.}, 64(5):1636--1667, 2004.

\bibitem[AFNT07]{AFNT07:model}
G.~Al{\`i}, V.~Furuholt, R.~Natalini, and I.~Torcicollo.
\newblock A mathematical model of sulphite chemical aggression of limestones
  with high permeability. {I}. {M}odeling and qualitative analysis.
\newblock {\em Transp. Porous Media}, 69(1):109--122, 2007.

\bibitem[BBR79]{BBR79}
A.E. Berger, H.~Brezis, and J.C.W Rogers.
\newblock A numerical method for solving the problem {$u_t-\Delta f(u)=0$}.
\newblock {\em RAIRO numerical analysis}, 13:297--312, 1979.

\bibitem[BLTR00]{BLR00}
R.~Bugini, M.~Laurenzi~Tabasso, and M.~Realini.
\newblock Rate of formation of black crusts on marble. {A} case study.
\newblock {\em J. Cultural Heritage}, pages 111--116, 2000.

\bibitem[BP72]{BP72}
H.~Br{\'e}zis and A.~Pazy.
\newblock Convergence and approximation of semigroups of nonlinear operators in
  {B}anach spaces.
\newblock {\em J. Functional Analysis}, 9:63--74, 1972.

\bibitem[CGN{\etalchar{+}}09]{CGNNS:teos}
F.~Clarelli, C.~Giavarini, R.~Natalini, C.~Nitsch, and M.L Santarelli.
\newblock Mathematical models for the consolidation processes in stones.
\newblock In {\em Proc. of ``{I}nternational {S}ymposium: {S}tone
  {C}onsolidation in {C}ultural {H}eritage - research and practice''.
  {L}isbona, {M}ay 2008.}, 2009.
\newblock to appear.

\bibitem[CL71]{CL71}
M.G. Crandall and T.M. Liggett.
\newblock Generation of {S}emi-{G}roups of non linear transformations on
  general {B}anach spaces.
\newblock {\em Amer. J. Math.}, 93:265--298, 1971.

\bibitem[CNPS07]{CNPS07:degdiff}
F.~Cavalli, G.~Naldi, G.~Puppo, and M.~Semplice.
\newblock High-order relaxation schemes for non linear degenerate diffusion
  problems.
\newblock {\em SIAM Journal on Numerical Analysis}, 45(5):2098--2119, 2007.

\bibitem[GKPC89]{GKPC89}
K.L. Gauri, N.P. Kulshreshtha, A.R. Punuru, and A.N. Chowdhury.
\newblock Rate of decay of marble in laboratory and outdoor exposure.
\newblock {\em J. Mater. Civil Eng.}, pages 73--85, 1989.

\bibitem[GN05]{GN05}
F.~R. Guarguaglini and R.~Natalini.
\newblock Global existence of solutions to a nonlinear model of sulphation
  phenomena in calcium carbonate stones.
\newblock {\em Nonlinear Anal. Real World Appl.}, 6(3):477--494, 2005.

\bibitem[GN07]{GN07}
F.~R. Guarguaglini and R.~Natalini.
\newblock Fast reaction limit and large time behavior of solutions to a
  nonlinear model of sulphation phenomena.
\newblock {\em Comm. Partial Differential Equations}, 32(1-3):163--189, 2007.

\bibitem[Gre97]{Greenbaum:book}
Anne Greenbaum.
\newblock {\em Iterative methods for solving linear systems}, volume~17 of {\em
  Frontiers in Applied Mathematics}.
\newblock Society for Industrial and Applied Mathematics (SIAM), Philadelphia,
  PA, 1997.

\bibitem[GSNF08]{GSNF08}
C.~Giavarini, M.L. Santarelli, R.~Natalini, and F.~Freddi.
\newblock A nonlinear model of sulphation of porous stones: numerical
  simulations and preliminary laboratory assessments.
\newblock {\em J. Cultural Heritage}, 9:14--22, 2008.

\bibitem[Hac85]{Hackbusch:book}
Wolfgang Hackbusch.
\newblock {\em Multigrid methods and applications}, volume~4 of {\em Springer
  Series in Computational Mathematics}.
\newblock Springer-Verlag, Berlin, 1985.

\bibitem[Hay82]{Hay82}
F.H. Hayne.
\newblock Deterioration of marble.
\newblock {\em Durability Build. Mater.}, (1):241--254, 1982.

\bibitem[Lip89]{Lip89}
W.T. Lipfert.
\newblock Atmospheric damage to calcareous stones: comparison and
  reconciliation of recent experimental findings.
\newblock {\em Atmos. Environ.}, 23:415--429, 1989.

\bibitem[MNV87]{MNV87}
E.~Magenes, R.~H. Nochetto, and C.~Verdi.
\newblock Energy error estimates for a linear scheme to approximate nonlinear
  parabolic problems.
\newblock {\em RAIRO Mod{\'e}l. Math. Anal. Num{\'e}r.}, 21(4):655--678, 1987.

\bibitem[OR70]{Ortega:book}
J.~M. Ortega and W.~C. Rheinboldt.
\newblock {\em Iterative solution of nonlinear equations in several variables}.
\newblock Academic Press, New York, 1970.

\bibitem[Saa03]{Saad:book}
Yousef Saad.
\newblock {\em Iterative methods for sparse linear systems}.
\newblock Society for Industrial and Applied Mathematics, Philadelphia, PA,
  second edition, 2003.

\bibitem[SC93]{Serra93:multi}
S.~Serra-Capizzano.
\newblock Multi-iterative methods.
\newblock {\em Comput. Math. Appl.}, 26(4):65--87, 1993.

\bibitem[SC06]{Serra06:glt}
Stefano Serra-Capizzano.
\newblock The {GLT} class as a generalized {F}ourier analysis and applications.
\newblock {\em Linear Algebra Appl.}, 419(1):180--233, 2006.

\bibitem[SDSC]{SDS:degdiff}
M.~Semplice, M.~Donatelli, and S.~Serra-Capizzano.
\newblock Multigrid and preconditioning strategies for implicit {PDE} solvers
  for degenerate parabolic equations.
\newblock {\em IMA}.
\newblock Submitted. Preprint arXiv:0907.2600v2.

\bibitem[SDSC09]{SDS:monum}
M.~Semplice, M.~Donatelli, and S.~Serra-Capizzano.
\newblock Preconditioned fully implicit pde solvers for degenerate parabolic
  equations with applications to monument conservation.
\newblock {\em http:{$\backslash\backslash$}www.arXiv.org}, (0907.2600v1), 15
  July 2009.

\bibitem[Til98]{Tilli98}
P.~Tilli.
\newblock Locally {T}oeplitz sequences: spectral properties and applications.
\newblock {\em Linear Algebra Appl.}, 278(1-3):91--120, 1998.

\bibitem[TOS01]{Trottenberg:book}
U.~Trottenberg, C.~W. Oosterlee, and A.~Sch{\"u}ller.
\newblock {\em Multigrid}.
\newblock Academic Press Inc., San Diego, CA, 2001.
\newblock With contributions by A. Brandt, P. Oswald and K. St{\"u}ben.

\bibitem[V\'07]{VazquezBOOK}
J.~L. V\'azquez.
\newblock {\em The porous medium equation}.
\newblock Oxford Mathematical Monographs. The Clarendon Press Oxford University
  Press, Oxford, 2007.
\newblock Mathematical theory.

\end{thebibliography}

\end{document}